\documentclass[11pt]{article}

\usepackage{color, amsmath, amssymb, amsthm, graphicx, bbm, footnote, mathtools}

\usepackage[top=1in, bottom=1in, left=1in, right=1in]{geometry}

\usepackage[toc,page,header]{appendix}
\usepackage{titletoc}
\usepackage{tikz}
\usepackage{titling}
\usepackage{enumitem}
\allowdisplaybreaks

\usepackage{setspace}
\usepackage{adjustbox}
\usepackage[colorlinks,allcolors=blue]{hyperref}
\usepackage[round]{natbib}  
\hypersetup{
	breaklinks=true,   
	colorlinks=true,   
	pdfusetitle=true,  
}

\usepackage{apptools}
\AtAppendix{\counterwithin{theorem}{section}}
\AtAppendix{\counterwithin{remark}{section}}

\newtheorem{theorem}{Theorem}
\newtheorem{definition}{Definition}
\newtheorem{proposition}{Proposition}
\newtheorem{corollary}{Corollary}
\newtheorem{lemma}{Lemma}
\newtheorem{remark}{Remark}
\newtheorem{assumption}{Assumption}

\newcommand{\ind}{{\perp\!\!\!\perp}}
\newcommand{\esssup}{{\mathrm{ess}\sup}}

\title{ \textsc{Anti-concentration of Suprema of Gaussian Processes and Gaussian Order Statistics}}
\author{Alexander Giessing\thanks{Department of Statistics, University of Washington, Seattle, WA. E-mail: giessing@uw.edu.}}
\date{\today}
\begin{document}
	\maketitle
	\begin{abstract}
	We derive, up to a constant factor, matching lower and upper bounds on the concentration functions of suprema of separable centered Gaussian processes and order statistics of Gaussian random fields.	
	These bounds reveal that suprema of separable centered Gaussian processes $\{X_u : u \in U\}$ exhibit the same anti-concentration properties as a single Gaussian random variable with mean zero and variance $\mathrm{Var}(\sup_{u \in U} X_u)$. To apply these results to high-dimensional statistical problems, it is therefore essential to understand the asymptotic behavior of $\mathrm{Var}(\sup_{u \in U} X_u)$ as the dimension or metric entropy of the index set $U$ increases. Consequently, we also derive lower and upper bounds on this quantity.
	\end{abstract}

\section{Introduction}\label{sec:Intro}
From the perspective of large deviations and measure concentration, the supremum of a separable centered Gaussian process $\{X_u : u \in U\}$ behaves asymptotically similar to a single Gaussian random variable with mean zero and variance $\sup_{u \in U} \mathrm{Var}(X_u)$~\citep[][]{landau1970supremum, marcus1972sample, borell1975brunn, ibragimov1976norms}.
In this paper, we show that this behavior extends to the concentration function of the supremum, albeit with a different variance that captures the local variation around the mode of the distribution rather than the global variation of the supremum. 

The study of concentration functions was initiated by~\cite{levy1954theorie}, 
who defined the \emph{concentration function} of a random variables $X$ as
\begin{align}\label{eq:intro-1}
	Q(X, \varepsilon) := \sup_{t \in \mathbb{R}} \mathbb{P}\left\{ t \leq X \leq t + \varepsilon \right\}, \quad \varepsilon > 0.
\end{align}
By examining the asymptotic behavior of concentration functions as $\varepsilon \rightarrow 0$, L{\'e}vy and his contemporaries obtained explicit rates of convergence of the distribution of sums of i.i.d. random variables to normal limit laws in L{\'e}vy and Kolmogorov distances~\citep[e.g.][]{kolmogorov1956two, kolmogorov1958sur, esseen1966kolmogorov, esseen1968concentration, kesten1969sharper, miroshnikov1983remarks, lecam1986asymptotic}.

Recent developments in theoretical computer science, combinatorics, random matrix theory, and high-dimensional statistics have renewed the interest in concentration functions, specifically in the form of so-called \emph{anti-concentration inequalities}~\citep[e.g.][]{carbery2001Distributional, nazarov2003MaximaPerimeter, klivans2008learning, rudelson2008littlewood, rudelson2009smallest, rudelson2014small, nourdin2009density,   tao2009inverse, mossel2010noise, goetze2014ExplicitRates, chernozhukov2015ComparisonAnti, bobkov2015concentration, chen2015error, chen2015multivariate, meka2016anti, goetze2019LargeBall, deng2020beyond, fox2021anti, kwan2023anticoncentration}. These inequalities provide non-asymptotic upper bounds on concentration functions.

Within the field of high-dimensional statistics, anti-concentration inequalities for the suprema of Gaussian processes and the maxima of Gaussian random fields have received particular attention, as they prove useful in establishing high-dimensional analogues to the classical CLTs and consistency of bootstrap procedures~\citep{chernozhukov2014AntiConfidenceBands, chernozhukov2015ComparisonAnti, chernozhukov2017Nazarov}. However, the existing anti-concentration inequalities require the finite-dimensional marginals of the Gaussian process to be non-degenerate and therefore do not apply to Gaussian processes indexed by the structured sets that naturally arise in many high-dimensional statistical problems, such as Euclidean balls, cross-polytopes, and cones. Progress in this area has been incremental and problem-specific~\citep{chernozhukov2021NearlyOptimalCLT, deng2020beyond, lopes2019BootstrappingSpectral, lopes2020bootstrapping, lopes2022central, lopes2022improved}.

In this paper we resolve this issue. We derive a two-sided anti-concentration inequality for suprema of separable centered Gaussian processes  $\{X_u : u \in U\}$ without imposing assumptions on the distribution of their finite-dimensional marginals. Lower and upper bounds of the inequality differ only by a multiplicative constant and are thus asymptotically optimal as $\varepsilon \rightarrow 0$. As a preliminary result of significant interest in itself we also obtain lower and upper bounds on the concentration function of order statistics of Gaussian random fields (Sections~\ref{subsec:AntiConcentrationIneuqalities} and~\ref{sec:Proof-AntiConcentration}). We recover the Gaussian anti-concentration inequalities by~\cite{chernozhukov2014AntiConfidenceBands, chernozhukov2015ComparisonAnti, chernozhukov2017Nazarov} and~\cite{deng2020beyond} as special cases from the new inequality. We show that if the one-dimensional marginals of the Gaussian process are (strongly) correlated, the new bounds on the concentration function can be substantially tighter than the previous ones (Section~\ref{subsec:Literature}). Since the new bounds depend on $\mathrm{Var}(\sup_{u \in U} X_u)$, it is crucial to understand the asymptotic behavior of this variance as the metric entropy or dimension of the index set $U$ increases. We therefore also derive lower and upper bounds on this quantity (Section~\ref{subsec:BoundsVariance}). 

\section{Anti-concentration inequalities for Gaussian random fields and Gaussian processes}\label{subsec:AntiConcentrationIneuqalities}
We begin with the main results of this paper: anti-concentration inequalities for order statistics of finite-dimensional Gaussian random fields and suprema of infinite-dimensional separable Gaussian processes.

An \emph{($n$-dimensional) Gaussian random field} is a collection $X = \{X_i : 1 \leq i \leq n\}$ of $n$ random variables such that each finite subset of this collection is distributed according to a multivariate normal law. We call a Gaussian random field \emph{centered} if $\mathbb{E}[X_i] = 0$ for all $1 \leq i \leq n$ and denote by $X_{(1)} \leq  \ldots X_{(n)}$ its order statistics in non-decreasing order. For such Gaussian random fields we have the following anti-concentration result:

\begin{theorem}\label{theorem:AntiConcentration-OrderStatistics}
	Let $X = \{X_i : 1 \leq i \leq n\}$ be a Gaussian random field such that $\mathrm{Var}(X_i)  > 0$ and $\mathrm{Cor}(X_i, X_j) < 1$ for all $1 \leq i \neq j \leq n$. Denote the $k^{th}$ order statistic of the $X_i$'s by $X_{(k)}$. Then, for all $\varepsilon > 0$ and $1 \leq k \leq n$,
	\begin{align*}
		\frac{\varepsilon/\sqrt{12}}{ \sqrt{\mathrm{Var}\left(X_{(k)}\right) + \varepsilon^2/ 12}} \leq Q\left(X_{(k)},\varepsilon  \right) \leq \frac{\varepsilon\sqrt{12}}{ \sqrt{\mathrm{Var}\left(X_{(k)}\right) + \varepsilon^2/ 12}}.
	\end{align*}
	The result remains true when $X_{(k)}$ is replaced by the $k^{th}$ order statistic of the $|X_i|$'s.
\end{theorem}
\begin{remark}\label{remark:theorem:AntiConcentration-OrderStatistics}
	The lower bound holds for arbitrary random variables, not just for order statistics of Gaussian random fields.
\end{remark}
We defer the proof of this theorem to Section~\ref{sec:Proof-AntiConcentration}. By specializing to the maximum of a centered Gaussian random field, we can remove the conditions on the marginal and joint distribution:

\begin{corollary}\label{corollary:theorem:AntiConcentration-OrderStatistics}
	Let $X = \{X_i : 1 \leq i \leq n\}$ be a centered Gaussian random field. Set $\widetilde{Z}_n = \max_{1 \leq i \leq n} |X_i|$. Then, for all $\varepsilon > 0$,
	\begin{align*}
		\frac{\varepsilon/\sqrt{12}}{\sqrt{\mathrm{Var}(\widetilde{Z}_n) + \varepsilon^2/12}} \leq  Q(\widetilde{Z}_n, \varepsilon) \leq \frac{\varepsilon}{\sqrt{\mathrm{Var}(\widetilde{Z}_n) + \varepsilon^2/12}}.
	\end{align*}
\end{corollary}

\begin{proof}[Proof of Corollary~\ref{corollary:theorem:AntiConcentration-OrderStatistics}]
	Set $\widetilde{X}_i = X_i$ and $\widetilde{X}_{n +i} = -X_i$ for $ 1\leq i \leq n$. Then $\widetilde{Z}_n = \max_{1 \leq i \leq n} |X_i| \equiv \max_{1 \leq i \leq 2n} \widetilde{X}_i$. 	Without loss of generality we may assume that $\mathrm{Var}(X_i)  > 0$ and $\mathrm{Cor}(\widetilde{X}_j, \widetilde{X}_k) < 1$ for all $1 \leq i \leq n$ and $1 \leq j \neq k \leq 2n$. Indeed, if there exists $X_j$ with $\mathrm{Var}(X_j) = 0$, then $X_j = 0$ almost surely and thus $\widetilde{Z}_n = \max_{1 \leq i \leq n, i \neq j} |X_i|$ almost surely. Similarly, if there exist perfectly correlated $\widetilde{X}_i, \widetilde{X}_j$ with $1 \leq i \neq j \leq 2n$, then we can delete on of them, say $\widetilde{X}_j$, and $\widetilde{Z}_n = \max_{1 \leq i \leq 2n, i \neq j} \widetilde{X}_i$ almost surely. Hence, the claim follows from Theorem~\ref{theorem:AntiConcentration-OrderStatistics}. If $\mathrm{Var}(X_i) = 0$ for all $1 \leq i \leq n$, then $ Q(\widetilde{Z}_n, \varepsilon)  = 1$ and the left inequality is sharp while the right inequality is trivial.
\end{proof}

Next, we lift Corollary~\ref{corollary:theorem:AntiConcentration-OrderStatistics} from $n$-dimensional Gaussian random fields to infinite-dimensional Gaussian processes. To this end, we introduce the following additional notation and concepts. Let $X = \{X_u : u \in U\}$ be a generic stochastic process defined on some probability space $(\Omega, \mathcal{A}, \mathbb{P})$ and let $(U,d)$ be a pseudo-metric space. $X$ is \emph{separable} if there exists $U_0 \subset U$, $U_0$ countable, and $\Omega_0 \subset \Omega$ with $\mathbb{P}\{\Omega_0\} = 1$ such that $X_u(\omega) \in \overline{X_v(\omega) : v \in U_0 \cap B_d(u, \varepsilon)}$ for all $\omega \in \Omega_0$, $u \in U$ and $\varepsilon > 0$, where $B_d(u, \varepsilon)$ is the open $d$-ball about $u$ with radius $\varepsilon$.
Consequently, if $X$ is separable, then $(U,d)$ is separable. Moreover, if $X$ is separable, then $\sup_{u \in U} X_u = \sup_{u \in U_0} X_u$ a.s., and the latter, being a countable supremum, is measurable. The same holds for $|X|$. $X$ is a \emph{centered Gaussian process} if its finite-dimensional marginals are multivariate normal with mean zero.

\begin{theorem}\label{theorem:AntiConcentration-SeparableProcess}
	Let $X = \{X_u : u \in U\}$ be a separable centered Gaussian process. Set $\widetilde{Z} = \sup_{u \in U}|X_u|$ and assume that $\widetilde{Z} < \infty$ a.s. For all $\varepsilon > 0$,
	\begin{align*}
		\frac{\varepsilon/\sqrt{12}}{ \sqrt{\mathrm{Var}(\widetilde{Z}) + \varepsilon^2/ 12}} \leq Q(\widetilde{Z}, \varepsilon) \leq \frac{\varepsilon\sqrt{12}}{ \sqrt{\mathrm{Var}(\widetilde{Z}) + \varepsilon^2/ 12}}.
	\end{align*}
\end{theorem}
\begin{remark}\label{remark:theorem:AntiConcentration-SeparableProcess-1}
	The same inequalities hold for $Z = \sup_{u \in U} X_u$, where $X = \{X_u : u \in U\}$ is a separable (not necessarily centered) Gaussian process with non-degenerate one-dimensional marginals, i.e. $\inf_{u \in U} \mathrm{Var}(X_u) > 0$.
\end{remark}

\begin{remark}\label{remark:theorem:AntiConcentration-SeparableProcess-2}
	Note that $a \leq \mathrm{erf}(a)\sqrt{1 + a^2 } \leq \sqrt{2}a$ for all $a \geq 0$ and $\mathrm{erf}(a) = 2/\sqrt{\pi} \int_0^a e^{-t^2} dt$.  Therefore, if $Z \sim N(0, \sigma^2)$, then
	\begin{align*}
		\frac{\varepsilon/\sqrt{2}}{\sqrt{\sigma^2 + \varepsilon^2/2}} \leq Q(|Z|, \varepsilon) = \mathrm{erf}\left(\varepsilon/\sqrt{2 \sigma^2}\right) 
		\leq \frac{\varepsilon}{\sqrt{\sigma^2 + \varepsilon^2/2}}, \quad \varepsilon > 0.
	\end{align*}
	Thus, suprema of separable centered Gaussian processes have essentially the same anti-concentration properties as a single Gaussian random variable with mean zero and variance $\mathrm{Var}(\sup_{u \in U} |X_u|)$. This is reminiscent of (and yet different to) the classical large deviation and measure concentration results which imply that the tails of suprema of a separable centered Gaussian process are asymptotically equivalent to a single Gaussian random variable with mean zero and variance  $\sup_{u \in U}\mathrm{Var}(X_u)$. 
\end{remark}

\begin{proof}[Proof of Theorem~\ref{theorem:AntiConcentration-SeparableProcess}]
	We only need to demonstrate the upper bound, the lower bound follows from Remark~\ref{remark:theorem:AntiConcentration-OrderStatistics}. Without loss of generality, we may assume that $\mathrm{Var}(\widetilde{Z}) > 0$. If not, the upper bound is trivial.
	Since $X$ is separable, there exists a sequence of finite sets $U_n \subseteq U_0$ such that $\widetilde{Z}_n := \max_{u \in U_n} |X_u| \rightarrow \sup_{u \in U_0}|X_u| =  \sup_{u \in U}|X_u| =: \widetilde{Z}$ almost surely as $n \rightarrow \infty$; hence $\widetilde{Z}_n \rightarrow \widetilde{Z}$ weakly.
	Since $\mathrm{Var}(\widetilde{Z}) > 0$, the supremum $\widetilde{Z}_n$ is not degenerate at 0 and the distribution function $ t \mapsto \mathbb{P}\{ \widetilde{Z} \leq t\} $ is continuous for all $ t\geq 0$~\citep[][]{gaenssler2007continuity}. Therefore, by the reverse triangle inequality, for all $\varepsilon > 0$,
	\begin{align}\label{eq:theorem:AntiConcentration-SeparableProcess-1}
		\left| Q(\widetilde{Z}_n, \varepsilon) - Q(\widetilde{Z}, \varepsilon) \right|\leq 2 \sup_{t \geq 0} \left| \mathbb{P}\{ \widetilde{Z}_n \leq t\} - \mathbb{P}\{ \widetilde{Z} \leq t\} \right| \rightarrow 0 \quad{} \mathrm{as} \quad{} n \rightarrow \infty.
	\end{align}	 
	Moreover, by the reverse Lyapunov inequality for suprema of Gaussian processes 
	and the dominated convergence theorem, $\mathrm{Var}(\widetilde{Z}_n) \rightarrow \mathrm{Var}(\widetilde{Z})$ as $n \rightarrow \infty$. Hence, by Corollary~\ref{corollary:theorem:AntiConcentration-OrderStatistics}, for all $\varepsilon > 0$,
	\begin{align}\label{eq:theorem:AntiConcentration-SeparableProcess-2}
		\lim_{n \rightarrow \infty} Q(\widetilde{Z}_n, \varepsilon) \leq \lim_{n \rightarrow \infty}	\frac{\varepsilon\sqrt{12}}{ \sqrt{\mathrm{Var}(\widetilde{Z}_n) + \varepsilon^2/ 12}} = 	\frac{\varepsilon\sqrt{12}}{ \sqrt{\mathrm{Var}(\widetilde{Z}) + \varepsilon^2/ 12}}.
	\end{align}
	To conclude the proof, combine~\eqref{eq:theorem:AntiConcentration-SeparableProcess-1} and~\eqref{eq:theorem:AntiConcentration-SeparableProcess-2}.
\end{proof}

There exists an equivalent formulation of Theorem~\ref{theorem:AntiConcentration-SeparableProcess} in terms of the mode of the density of $\widetilde{Z}$:
\begin{corollary}\label{corollary:AntiConcentration-SeparableProcess-MaxDensity}
	Recall assumptions and notation from Theorem~\ref{theorem:AntiConcentration-SeparableProcess}. Set $M(\widetilde{Z}) = \esssup_{z \in \mathbb{R}} f_{\widetilde{Z}}(z)$, where $f_{\widetilde{Z}}(z)$ is the density of $\widetilde{Z}$ w.r.t. the Lebesgue measure on the real line. For all $\varepsilon > 0$,
	\begin{align*}
		\frac{\varepsilon M(\widetilde{Z})}{\sqrt{ 12^2 + \varepsilon^2M^2(\widetilde{Z})}} \leq  Q(\widetilde{Z}, \varepsilon) \leq \frac{12 \varepsilon M(\widetilde{Z})}{ \sqrt{1 + \varepsilon^2M^2(\widetilde{Z})}}.
	\end{align*}
\end{corollary}
\begin{remark}\label{remark:corollary:AntiConcentration-SeparableProcess-MaxDensity}
	The same inequalities hold for $Z = \sup_{u \in U} X_u$, where $X = \{X_u : u \in U\}$ is a separable (not necessarily centered) Gaussian process with non-degenerate one-dimensional marginals, i.e. $\inf_{u \in U} \mathrm{Var}(X_u) > 0$.
\end{remark}

\begin{proof}[Proof of Corollary~\ref{corollary:AntiConcentration-SeparableProcess-MaxDensity}]
	Combine Theorem~\ref{theorem:AntiConcentration-SeparableProcess} and below Proposition~\ref{proposition:S-Concavity-Affine-Invariance} with $s = -1/6$ (see Section~\ref{subsec:Proof-AntiConcentration-OrderStatistics}).
\end{proof}

From a purely mathematical perspective the two-sided inequalities on the concentration functions in Corollary~\ref{corollary:theorem:AntiConcentration-OrderStatistics} and Theorem~\ref{theorem:AntiConcentration-SeparableProcess} are quite satisfying since the, up to a multiplicative constant, matching lower and upper bounds imply that they are asymptotically optimal as $\varepsilon \rightarrow 0$. However, from a statistical perspective these results are only useful, if we can provide tight bounds on the variances $\mathrm{Var}(\max_{1 \leq i \leq n} |X_i|)$ and $\mathrm{Var}(\sup_{u \in U} |X_u|)$, respectively. We address this issue in Section~\ref{subsec:BoundsVariance}.

\section{Recovering and improving results from the literature}\label{subsec:Literature}
Our results from the preceding section can be used to recover and improve the Gaussian anti-concentration inequalities by~\cite{chernozhukov2014AntiConfidenceBands, chernozhukov2015ComparisonAnti, chernozhukov2017Nazarov} and~\cite{deng2020beyond}, which are the most commonly used anti-concentration inequalities in high-dimensional statistics. We do not discuss anti-concentration inequalities derived from small ball probabilities~\cite[e.g.,][]{rudelson2008littlewood, rudelson2009smallest, rudelson2014small, goetze2014ExplicitRates, goetze2019LargeBall}, since these authors use a slightly different definition of a concentration function that only applies to finite-dimensional Gaussian random fields. For a comprehensive review of the literature on (Gaussian) anti-concentration we refer to~\cite{chernozhukov2015ComparisonAnti}. The main lesson from this section is that the anti-concentration inequalities by~\cite{chernozhukov2014AntiConfidenceBands, chernozhukov2015ComparisonAnti, chernozhukov2017Nazarov} and~\cite{deng2020beyond} fail to accurately capture the effect of the correlation between the one-dimensional marginals of the Gaussian process. This typically results in conservative (or even vacuous) upper bounds on the concentration function when compared to our new inequalities.

For $n$-dimensional Gaussian random fields Corollary~\ref{corollary:AntiConcentration-SeparableProcess-MaxDensity} readily implies the following refined version of Theorem 10 by~\cite{deng2020beyond}: 

\begin{corollary}\label{corollary:AntiConcentration-CountableProcess}
	Let $X = \{X_i : 1 \leq i \leq n\}$ be a Gaussian random field with $\mathrm{Var}(X_i) = \sigma_i^2$ for all $1 \leq i \leq n$ and $\sigma_{(1)}^2 \leq \ldots \leq \sigma_{(n)}^2$ be the ordered values of $\sigma_1^2, \ldots, \sigma_n^2$. Set $Z_n = \max_{1 \leq i \leq n} x_i$. Then, for all $\varepsilon > 0$,
	\begin{align*}
	\frac{\varepsilon/\sqrt{12}}{ \sqrt{\sigma_{(n)}^2 + \varepsilon^2/ 12}} \leq Q(Z_n, \varepsilon) \leq \frac{  12 \left(2 + \sqrt{2\log n}\right) \varepsilon}{\sqrt{\bar{\sigma}_n^2 +\left(2 + \sqrt{2\log n}\right)^2  \varepsilon^2} },
	\end{align*}
	where $\bar{\sigma}_n = \left(2 + \sqrt{2\log n}\right) / \left(1/ \sigma_{(1)} + \max_{1 \leq k \leq n} \left(1 + \sqrt{2\log k}\right)/\sigma_{(k)}\right) \geq \sigma_{(1)}$.
\end{corollary}

\begin{proof}[Proof of Corollary~\ref{corollary:AntiConcentration-CountableProcess}]
	The lower bound on the concentration function follows from Theorem~\ref{theorem:AntiConcentration-OrderStatistics} and the well-known fact that $\mathrm{Var}(Z_n) \leq \sigma_{(n)}^2$ (see also Proposition~\ref{theorem:UpperBoundVariance-SeparableProcess}). Next, consider the upper bound on the concentration function. By Theorem 10 in~\cite{deng2020beyond} for all $\varepsilon > 0$ and all $t \in \mathbb{R}$, we have
	\begin{align*}
		\varepsilon^{-1} \mathbb{P}\left\{ t \leq  \max_{1 \leq i \leq n} X_i \leq t + \varepsilon  \right\} \leq 1/ \sigma_{(1)} + \max_{1 \leq k \leq n} \left(1 + \sqrt{2\log k}\right)/\sigma_{(k)}.
	\end{align*}
	Take $\varepsilon \downarrow 0$ to conclude that
	\begin{align*}
		M(Z_n) \leq 1/ \sigma_{(1)} + \max_{1 \leq k \leq n} \left(1 + \sqrt{2\log k}\right)/\sigma_{(k)}.
	\end{align*}
	Notice that the map $x \mapsto 12 x/\sqrt{1 + x^2}$ is increasing in $x$. Hence, the claim follows from Corollary~\ref{corollary:AntiConcentration-SeparableProcess-MaxDensity}, Remark~\ref{remark:corollary:AntiConcentration-SeparableProcess-MaxDensity}, and above upper bound on $M(Z_n)$.
\end{proof}

From Corollary~\ref{corollary:AntiConcentration-CountableProcess} we deduce a slightly refined version of Nazorov's inequality in~\cite{chernozhukov2017Nazarov}:
\begin{corollary}\label{corollary:AntiConcentration-CountableProcess-2}
	Recall the notation and setup of Corollary~\ref{corollary:AntiConcentration-CountableProcess}. For all $x \in \mathbb{R}^n$ and $\varepsilon > 0$,
	\begin{align*}
		\mathbb{P}\left\{ x \leq X \leq x + \varepsilon \mathbf{1}  \right\} \leq \frac{  12 \left(2 + \sqrt{2\log n}\right) \varepsilon}{\sqrt{\bar{\sigma}_n^2 +\left(2 + \sqrt{2\log n}\right)^2  \varepsilon^2} },
	\end{align*}
	where $\mathbf{1}=(1, \ldots, 1)' \in \mathbb{R}^n$ and the vector-valued inequality $x \leq X \leq x + \varepsilon \mathbf{1}$ is understood element-wise.
\end{corollary}

\begin{proof}[Proof of Corollary~\ref{corollary:AntiConcentration-CountableProcess-2}]
	By Corollary~\ref{corollary:AntiConcentration-CountableProcess}, for all $\varepsilon > 0$,
	\begin{align*}
		\sup_{t \in \mathbb{R}} \mathbb{P}\left\{ t \leq  \max_{1 \leq i \leq n} X_i \leq t + \varepsilon  \right\} \leq \frac{  12 \left(2 + \sqrt{2\log n}\right) \varepsilon}{\sqrt{\bar{\sigma}_n^2 +\left(2 + \sqrt{2\log n}\right)^2  \varepsilon^2} }.
	\end{align*}
	The probability on the left-hand side in above display is clearly larger than
	\begin{align*}
		\sup_{x \in \mathbb{R}^n} \mathbb{P}\left\{ \forall i \in [n] : x_i \leq  X_i \leq x_i + \varepsilon  \right\} = \sup_{x \in \mathbb{R}^n} \mathbb{P}\left\{ x \leq X \leq x + \varepsilon \mathbf{1}  \right\}.
	\end{align*}
\end{proof}

The upper bounds in Corollaries~\ref{corollary:AntiConcentration-CountableProcess} and~\ref{corollary:AntiConcentration-CountableProcess-2} both depend explicitly on the dimension $n$ of the Gaussian random field and become trivial as $n \rightarrow \infty$. It is, therefore, intuitively obvious that these bounds must be incorrect. The following two cases illustrate the source of the problem with these bounds. First, suppose that the one-dimensional marginals $X_1, \ldots, X_n$ are perfectly correlated. In this case, $\max_{1 \leq i \leq n}X_i \overset{d}{=} X_1$ and by Remark~\ref{remark:theorem:AntiConcentration-SeparableProcess-2} the upper bound on the concentration function in Corollary~\ref{corollary:AntiConcentration-CountableProcess} should be proportional to $\varepsilon/(\sigma_1^2 + \varepsilon^2/2)^{1/2}$ and independent of $n$. Second, suppose that the one-dimensional marginals $X_1, \ldots, X_n$ are uncorrelated. In this case, $\sup_{x \in \mathbb{R}^n}\mathbb{P}\left\{ x \leq X \leq x + \varepsilon \mathbf{1}  \right\} =  \left(\sup_{t \in \mathbb{R}} \mathbb{P}\left\{ t \leq X_1 \leq t + \varepsilon \right\} \right)^n \leq ( \varepsilon^2/(\sigma_1^2 + \varepsilon^2/2))^{n/2}$. Hence, the upper bound on the concentration function in Corollary~\ref{corollary:AntiConcentration-CountableProcess-2} is too pessimistic. In fact, the probability in Corollary~\ref{corollary:AntiConcentration-CountableProcess-2} is better interpreted as a small ball probability rather than a concentration function for the maximum of a Gaussian random field. Therefore, the exponentially fast decay in the dimension should not surpris~\citep{kuelbs1993metric, li1999approximation}.

Another, slightly less trivial example that showcases how our results improve the upper bound of  Corollary~\ref{corollary:AntiConcentration-CountableProcess} is an equicorrelated centered Gaussian random field. We call a centered Gaussian random field \emph{equicorrelated} if its one-dimensional marginals satisfy $X_i \sim N(0,1)$ and $\mathrm{Cov}(X_k, X_j) = \rho \in \left(-\frac{1}{n-1}, 1\right)$ for all $1 \leq i, k\neq j \leq n$. As usual, we set $Z_n = \max_{1 \leq i \leq n} X_i$. By~\cite{tanguy2017quelques} (Proposition 4.1.1) there exists an absolute constant $C_0 > 0$ such that for all $\rho \in [0,1)$,
\begin{align*}
	\rho \leq \mathrm{Var}(Z_n) \leq \frac{C_0}{\log n} + \rho.
\end{align*}
Thus, by Theorem~\ref{theorem:AntiConcentration-SeparableProcess} and Remark~\ref{remark:theorem:AntiConcentration-SeparableProcess-1} there exists $C_1 > 0$ such that for all $\rho \in [0,1)$ and $\varepsilon > 0$,
\begin{align}\label{eq:Literature-1}
	\frac{\varepsilon/C_1}{ \sqrt{\rho + (\log n)^{-1} + \varepsilon^2}} \leq Q(Z_n, \varepsilon) \leq \frac{C_1\varepsilon}{ \sqrt{\rho  + \varepsilon^2}}.
\end{align}
If $0 < \rho = O(1)$, above two-sided bound is a significant improvement over the bound in Corollary~\ref{corollary:AntiConcentration-CountableProcess}. Whereas, if $\rho = O(1/\log n)$, then above two-sided bounds and the bound in Corollary~\ref{corollary:AntiConcentration-CountableProcess} are of the same order. Thus, Corollary~\ref{corollary:AntiConcentration-CountableProcess} fails to capture the correlation structure between the one-dimensional marginals of $X$ and treats them as if they were (asymptotically) independent.

Next, for (infinite-dimensional) separable centered Gaussian processes Theorem~\ref{theorem:AntiConcentration-SeparableProcess} yields a version of Theorem 2.1 by~\cite{chernozhukov2014AntiConfidenceBands}:

\begin{corollary}\label{corollary:AntiConcentration-SeparableProcess}
	Let $X = \{X_u : u \in U\}$ be a separable centered Gaussian process such that $\inf_{u \in U} \mathrm{Var}(X_u) \geq \underline{\sigma}^2 > 0$. Set $Z = \sup_{u \in U} X_u$ and assume that $Z < \infty$ a.s. Then, $ 0 \leq \mathbb{E}[Z] < \infty$ and, for all $\varepsilon > 0$,
	\begin{align*}
		Q(Z, \varepsilon) \leq \frac{15\sqrt{12}  (\mathbb{E}[Z/\underline{\sigma}] +1) \varepsilon}{ \sqrt{\underline{\sigma}^2 + 15 (\mathbb{E}[Z/\underline{\sigma}]+1)^2  \varepsilon^2}}.
	\end{align*}
\end{corollary}

\begin{proof}[Proof of Corollary~\ref{corollary:AntiConcentration-SeparableProcess}]
	By Proposition~\ref{theorem:LowerBoundVariance-SeparableProcess} (see Section~\ref{subsec:BoundsVariance}) we have
	\begin{align*}
		\mathrm{Var}(Z/\underline{\sigma}) \geq \frac{1}{15^2} \left(\frac{1}{1 + \mathbb{E}[Z/\underline{\sigma}]}\right)^2,
	\end{align*}
	and, hence, by Theorem~\ref{theorem:AntiConcentration-SeparableProcess} and Remark~\ref{remark:theorem:AntiConcentration-SeparableProcess-1}, for all $\varepsilon > 0$,
	\begin{align*}
		Q(Z, \varepsilon) = Q(Z/\underline{\sigma}, \varepsilon/\underline{\sigma}) \leq \frac{15\sqrt{12}  (\mathbb{E}[Z/\underline{\sigma}] +1) \varepsilon}{ \sqrt{\underline{\sigma}^2 + 15 (\mathbb{E}[Z/\underline{\sigma}]+1)^2  \varepsilon^2}}.
	\end{align*}
\end{proof}

Since the map $x \mapsto x/\sqrt{1 + x^2}$ is monotone increasing in $x > 0$, we can replace $\mathbb{E}[Z/\underline{\sigma}]$ with any upper bound. This is extremely convenient since there exist many well-established techniques to obtain (tight) upper bounds on the expected value of suprema of Gaussian processes.

\section{Lower and upper bounds on the variance of surprema of Gaussian processes}\label{subsec:BoundsVariance}
In applications of Theorem~\ref{theorem:AntiConcentration-SeparableProcess} to statistical problems it is often not enough to simply state bounds in terms of $\mathrm{Var}(\sup_{u \in U}|X_u|)$. For instance, in high-dimensional statistics we usually observe that $\mathrm{Var}(\sup_{u \in U}|X_u|) \rightarrow 0$ as the metric entropy or dimension of the index set $U$ increases~\citep[e.g.,][]{chatterjee2014Superconcentration,  biau2015HighDimPNorms, giessing2023bootstrap}. In such situations, it is crucial to understand at which rate the variance of the supremum collapses. Thus far, most results in the literature have focused on bounds on the variance of (smooth) functions of $n$-dimensional spherical Gaussian random fields~\citep[e.g.,][]{cacoullus1982upper, talagrand1994Russo, chatterjee2014Superconcentration, ding2015multiple, tanguy2015some, tanguy2017quelques}. In this section, we  build on these results and extend them to (infinite-dimensional) Gaussian processes with arbitrary finite-dimensional marginals.

The bounds in this section apply to any separable centered Gaussian process. Readers interested in specific applications such as hypothesis testing of high-dimensional mean vectors, inference on the spectral norm of covariance matrices, or construction of uniform confidence bands for functions in reproducing Kernel Hilbert spaces, may consult our companion paper~\cite{giessing2023Gaussian}. If information about the geometry of the index set of the Gaussian process is available, the bounds in this section can be improved. For example, for Gaussian processes indexed by $\ell_p$-norm balls in $\mathbb{R}^n$ and $2 \leq p \ll \log n$ we have derived substantially tighter bounds via direct calculations~\citep{giessing2023bootstrap}.

The first proposition establishes a lower bound on the variance of suprema of Gaussian processes in terms of their expected value. This result generalizes Theorem 1.8 in~\cite{ding2015multiple} which only applies to the maximum of homogeneous $n$-dimensional Gaussian random fields. It is also key to the proof of Corollary~\ref{corollary:AntiConcentration-SeparableProcess}.

\begin{proposition}\label{theorem:LowerBoundVariance-SeparableProcess}
	Let $X = \{X_u : u \in U\}$ be a separable centered Gaussian process such that $\mathrm{Var}(X_u) \geq \underline{\sigma}^2 > 0$ for all $u \in U$. Set $Z = \sup_{u \in U}X_u$ and assume that $Z < \infty$ a.s. Then, $ 0 \leq \mathbb{E}[Z] < \infty$ and
	\begin{align*}
		\mathrm{Var}(Z) \geq \frac{1}{15^2} \left(\frac{\underline{\sigma}}{1 + \mathbb{E}[Z/\underline{\sigma}]}\right)^2.
	\end{align*}
	The result remains true if $Z$ is replaced by $\widetilde{Z} = \sup_{u \in U}|X_u|$. 
\end{proposition}

The proof of Proposition~\ref{theorem:LowerBoundVariance-SeparableProcess} requires only a few minor modifications of the original arguments by~\cite{ding2015multiple}. Nonetheless, for completeness, we provide the full argument below. The following technical result plays an important role.

\begin{lemma}[Lemma 2.2 in~\citeauthor{ding2015multiple},~\citeyear{ding2015multiple}]\label{lemma:ding2015multiple}
	Let $X = \{X_u : u \in U\}$ be a centered Gaussian process such that $\sup_{u \in U} \mathrm{Var}(X_u) < \infty$. Set $Z = \sup_{u \in U}X_u$ and assume that $Z < \infty$ a.s. For $t \in (0,1)$ set $U_t = \{u \in U: X_u \geq t \mathbb{E}[Z]\}$. Then, for $\lambda > 0$ arbitrary,
	\begin{align*}
		\mathbb{P} \left\{ \sup_{u \in U_t} X_u' \geq \sqrt{1- t^2} \mathbb{E}[Z] + \frac{\lambda}{\sqrt{1-t^2}} \right\} \leq \frac{\mathrm{Var}(Z)}{\lambda^2},
	\end{align*}
	where $X' = \{X_u' : u \in U\}$ is an independent copy of $X$.
\end{lemma}

\begin{proof}[Proof of Proposition~\ref{theorem:LowerBoundVariance-SeparableProcess}]
	Observe that, under the stated assumptions, $\mathbb{E}[Z] \geq \mathbb{E}[X_u] = 0$ and $\mathbb{E}[Z] < \infty$. Consider the case $\mathbb{E}[Z/\underline{\sigma}]\leq 1/2$. By Chebyshev's inequality,
	\begin{align}\label{eq:theorem:LowerBoundVariance-SeparableProcess-1}
		\mathrm{Var}\left(Z/\underline{\sigma} \right) &\geq \mathbb{P}\{Z/\underline{\sigma}  \geq 1\} \left(1 - \mathbb{E}[Z/\underline{\sigma}]\right)^2 \geq  \frac{1}{4} \mathbb{P}\{Z/\underline{\sigma}\geq 1\} \geq  \frac{1}{4}\inf_{u \in U} \mathbb{P}\{X_u/\underline{\sigma} \geq 1\}.
	\end{align}
	Then, by the lower bound on Mill's ratio for a standard normal random variable,
	\begin{align}\label{eq:theorem:LowerBoundVariance-SeparableProcess-2}
		\inf_{u \in U}\mathbb{P}\{X_u/\underline{\sigma} \geq 1\} \geq \inf_{u \in U}\mathbb{P}\{X_u/\sqrt{\mathbb{E}[X_u^2]} \geq 1\} \geq \frac{1}{2}\frac{e^{-1/2}}{\sqrt{2\pi}} \geq \frac{1}{9}.
	\end{align}
	Combine~\eqref{eq:theorem:LowerBoundVariance-SeparableProcess-1} and~\eqref{eq:theorem:LowerBoundVariance-SeparableProcess-2} and conclude that
	\begin{align}\label{eq:theorem:LowerBoundVariance-SeparableProcess-3}
		\mathrm{Var}\left(Z/\underline{\sigma}\right) \left(1 + \mathbb{E}[Z/\underline{\sigma}]\right)^2 \geq \mathrm{Var}\left(Z/\underline{\sigma} \right) \geq \frac{1}{36}.
	\end{align}
	
	Now, consider the case $\mathbb{E}[Z/\underline{\sigma}] > 1/2$. Set $t = \sqrt{1 - \left(2 \mathbb{E}[Z/\underline{\sigma}]\right)^{-2}}$ and note that
	\begin{align*}
		\sqrt{1 - t^2}\mathbb{E}[Z/\underline{\sigma}] + \frac{\left(4 \mathbb{E}[Z/\underline{\sigma}]\right)^{-1}}{\sqrt{1 - t^2}} = 1.
	\end{align*}
	Set $U_t = \{u \in U: X_u \geq t \mathbb{E}[Z]\}$. Whence, by Lemma~\ref{lemma:ding2015multiple},
	\begin{align}\label{eq:theorem:LowerBoundVariance-SeparableProcess-4}
		16\left(\mathbb{E}[Z/\underline{\sigma}]\right)^{2}\mathrm{Var}\left/\underline{\sigma}\right) \geq \mathbb{P}\left\{ \sup_{u \in U_t} |Y_u|/\underline{\sigma} \geq 1 \right\}, 
	\end{align}
	where $Y = \{Y_u : u \in U\}$ is an independent copy of $X = \{X_u : u \in U \}$. To lower bound the probability on the right hand side in above display, intersect it with the event $\{U_t \neq \varnothing\}$, i.e.
	\begin{align*}
		\mathbb{P}\left\{ \sup_{u \in U_t} Y_u/\underline{\sigma} \geq 1 \right\} &\geq \mathbb{P} \left\{ \sup_{u \in U_t} Y_u/\underline{\sigma} \geq 1 \mid U_t \neq \varnothing \right\} \mathbb{P} \left\{U_t \neq \varnothing \right\} \\
		&\geq \inf_{u \in U} \mathbb{P} \left\{ Y_u/\underline{\sigma} \geq 1 \mid U_t \neq \varnothing \right\} \mathbb{P} \left\{ U_t \neq \varnothing \right\}.
	\end{align*}
	Since $Y \ind X$ and $U_t$ is a function of $X$ only, we have, as in~\eqref{eq:theorem:LowerBoundVariance-SeparableProcess-2},
	\begin{align}\label{eq:theorem:LowerBoundVariance-SeparableProcess-5} 
		\inf_{u \in U}\mathbb{P} \left\{  Y_u/\underline{\sigma} \geq 1 \mid U_t \neq \varnothing \right\} = \inf_{u \in U} \mathbb{P} \left\{ Y_u/\underline{\sigma} \geq 1 \right\}	\geq \inf_{u \in U }\mathbb{P}\{ Y_u/\sqrt{\mathbb{E}[Y_u^2]} \geq 1\} \geq \frac{1}{9}.
	\end{align}
	Moreover, since $\sqrt{a^2 - b^2} \leq a - b^2/(2a)$ for all $a \in (0, \infty)$ and $b \in (-a, a)$, one easily verifies that 
	\begin{align*}
		t \leq 1 - (2\mathbb{E}[Z/\underline{\sigma}])^{-2}/2.
	\end{align*}
	Therefore, by Cantelli's inequality,	
	\begin{align}\label{eq:theorem:LowerBoundVariance-SeparableProcess-6}
		\mathbb{P} \left\{ U_t \neq \varnothing \right\} = \mathbb{P} \left\{ Z/\underline{\sigma} \geq t \mathbb{E}[Z/\underline{\sigma}] \right\}  &\geq \mathbb{P} \left\{ Z/\underline{\sigma} \geq \mathbb{E}[Z/\underline{\sigma}] - (\mathbb{E}[Z/\underline{\sigma}])^{-1}/8 \right\} \nonumber \\
		&\geq 1 - \frac{64 (\mathbb{E}[Z/\underline{\sigma})^2\mathrm{Var}(Z/\underline{\sigma})}{64  (\mathbb{E}[Z/\underline{\sigma})^2 \mathrm{Var}(Z/\underline{\sigma}) + 1} \nonumber\\
		&\geq 1 - 64  (\mathbb{E}[Z/\underline{\sigma})^2 \mathrm{Var}(Z/\underline{\sigma}).
	\end{align}
	Combine~\eqref{eq:theorem:LowerBoundVariance-SeparableProcess-4}--\eqref{eq:theorem:LowerBoundVariance-SeparableProcess-6} to conclude that
	\begin{align}\label{eq:theorem:LowerBoundVariance-SeparableProcess-7}
		\mathrm{Var}\left(Z/\underline{\sigma}\right) \left(1 + \mathbb{E}[Z/\underline{\sigma}\right)^2 \geq \mathrm{Var}(Z/\underline{\sigma}) (\mathbb{E}[Z/\underline{\sigma}])^2\geq \frac{1}{208}.
	\end{align}
	Hence,~\eqref{eq:theorem:LowerBoundVariance-SeparableProcess-3} and~\eqref{eq:theorem:LowerBoundVariance-SeparableProcess-7} imply
	\begin{align*}
		\mathrm{Var}(Z) \geq \frac{1}{15^2}\left(\frac{\underline{\sigma}^2}{\underline{\sigma}  + \mathbb{E}[Z]}\right)^2.
	\end{align*}
\end{proof}

A natural follow-up question is how tight the lower bound in Proposition~\ref{theorem:LowerBoundVariance-SeparableProcess} is and whether there is room for improvement. The next proposition provides an answer to this question: Namely, if the one-dimensional marginals of the Gaussian process are (nearly) uncorrelated, then the lower bound on the variance in Theorem~\ref{theorem:LowerBoundVariance-SeparableProcess} is essentially optimal. However, if the one-dimensional marginals are strongly correlated, then there is a gap between the lower and the upper bound (see also Section~\ref{subsec:Literature}, discussion of equicorrelated Gaussian random fields).

\begin{proposition}\label{theorem:UpperBoundVariance-SeparableProcess}
	Let $X = \{X_u : u \in U\}$ be a separable centered Gaussian process such that $\mathrm{Var}(X_u) \leq \bar{\sigma}^2 < \infty$, and $|\mathrm{Cor}(X_u, X_v)| \leq \rho$ for all $u, v \in U$. Set $Z = \sup_{u \in U}X_u$ and assume that $Z < \infty$ a.s. Then,
	\begin{align*}
		\mathrm{Var}\left(Z\right) \leq  2^2\bar{\sigma}^2 \wedge \left[ 11^2\bar{\sigma}^2\rho + 60^2 \left(\frac{\bar{\sigma}}{(\mathbb{E}[Z/\bar{\sigma}] - 3 \sqrt{\log 2})_+} \right)^2 \right],
	\end{align*}
	where $(a)_+ = \max\{a, 0\}$ for all $a \in \mathbb{R}$. The result remains true if $Z$ is replaced by $\widetilde{Z} = \sup_{u \in U}|X_u|$.
\end{proposition}
\begin{remark}
	The expression $(\mathbb{E}[Z/ \bar{\sigma}] -  3 \sqrt{\log 2})_+ = \max\{ \mathbb{E}[Z/ \bar{\sigma}] - 3 \sqrt{\log 2}, 0 \}$ appears to be an artifact of our proof strategy. In high-dimensional applications we typically observe $\mathbb{E}[Z/ \bar{\sigma}] \rightarrow \infty$ as the dimension increase; hence, this detail is irrelevant for all practical purposes.
\end{remark}

The proof of this result relies on a technical lemma, which is a modest generalization of Theorem 5 in~\cite{tanguy2015some} from a Gaussian random vector with unit variances to a Gaussian random vector with arbitrary variances. Again, for completeness we include a proof.

\begin{lemma}\label{lemma:Tanguy2015some-Heteroscedastic}
	Let $X = (X_i)_{i=1}^n$ be a centered multivariate Gaussian random vector with covariance matrix $\Sigma = (\sigma_{ij})_{i,j=1}^n$ and $\max_{1 \leq i \leq n} \sigma_{ii} \leq \bar{\sigma}^2$. Suppose that there exists $r_0 \geq 0$ and a covering $\mathcal{C}(r_0)$ of $[n]$ such that for all indices $i,j \in [n]$ with $|\sigma_{ij}| \geq r_0$ there exists $D \in \mathcal{C}(r_0)$ with $i,j \in D$. Set $I = \arg\max_{1 \leq i\leq n} X_i$ and define $Z_n = X_I = \max_{1 \leq i \leq n} X_i$ and $\rho(r_0) =\max_{D \in \mathcal{C}(r_0)} \mathrm{P} ( I \in D)$. Further, suppose that there exists $\nu(r_0) > 0$ such that $\sum_{D \in \mathcal{C}(r_0)} \mathbf{1}\{I \in D\} \leq \nu(r_0)$ a.s. Then, for all $\theta \in \mathbb{R}$,
	\begin{align*}
		\mathrm{Var}\left(e^{\theta Z_n/2}\right) \leq \frac{\theta^2}{4}\left(r_0  + \frac{2 \bar{\sigma}^2\nu(r_0) }{\log\big(1/\rho(r_0)\big)}\right) \mathbb{E}\left[e^{\theta Z_n}\right].
	\end{align*}
	Moreover, for all $t \geq 0$,
	\begin{align*}
		\mathbb{P} \left\{ |Z_n - \mathbb{E}[Z_n]| \geq t\right\} \leq 6e^{-t/\sqrt{K}},
	\end{align*}
	where $K = 9 \left(r_0 + \frac{2 \bar{\sigma}^2\nu(r_0)}{\log(1/\rho(r_0))}\right)$.
\end{lemma}
\begin{proof}[Proof of Lemma~\ref{lemma:Tanguy2015some-Heteroscedastic}]
	We only sketch the key steps of the proof since the overall strategy is identical to the proof of Theorem 5~\cite{tanguy2015some}. The covariance identity on p. 114 in~\cite{ledoux2001concentration} implies that the variance of a smooth function $f: \mathbb{R}^n \rightarrow \mathbb{R}$ evaluated at $X \sim N(0, \Sigma)$ satisfies
	\begin{align*}
		\mathrm{Var}(f(X)) = \int_0^\infty e^{-t} \mathbb{E}[ (\nabla f(X))' \Sigma \nabla f(X^t)] dt,
	\end{align*}
	where $X^t = e^{-t}X + \sqrt{1 - e^{-t}} Y$ and $Y$ is an independent copy of $X$. Evaluating this identity with $f(X) = e^{\theta Z_n/2}$ yields
	\begin{align}\label{eq:lemma:Tanguy2015some-Heteroscedastic-1}
		\mathrm{Var}\left(e^{\theta Z_n/2}\right) = \frac{\theta^2}{4}\int_0^\infty e^{-t} \mathbb{E}\left[\sum_{i,j} \sigma_{ij} e^{\theta Z_n/2}\mathbf{1}\{I = i\} e^{\theta Z_n^t/2}\mathbf{1}\{I^t = j\}\right]dt,
	\end{align}
	where $I^t = \arg\max_{1 \leq i \leq n} X_i^t$ and $Z_n^t = X_{I^t}^t = \max_{1 \leq i \leq n} X_i^t$. Following the arguments in~\cite{tanguy2015some} on p. 242, we find that the integrand $\mathcal{I}$ in above identity can be upper bounded by
	\begin{align*}
		\mathcal{I} &=\mathbb{E}\left[\sum_{i,j} \sigma_{ij} e^{\theta Z_n/2}\mathbf{1}\{I = i\} e^{\theta Z_n^t/2}\mathbf{1}\{I^t = j\}\right] \nonumber\\
		&=  \mathbb{E}\left[e^{\theta (Z_n + Z_n^t)/2} \sigma_{I I^t}\right] \nonumber \\
		&\leq  \mathbb{E}\left[e^{\theta (Z_n + Z_n^t)/2} |\sigma_{I I^t}|\right] \nonumber \\
		&= \sum_{k \geq 0} \mathbb{E}\left[ e^{\theta (Z_n + Z_n^t)/2}  |\sigma_{I I^t}| \mathbf{1}\left\{2^{-k-1} \leq |\sigma_{II^t}|/\bar{\sigma}^2  \leq 2^{-k} \right\} \right].
	\end{align*}
	We choose $k_0 = \min\{ k \geq 0 : r_0 \geq 2^{-k-1} \bar{\sigma}^2\}$ and upper bound the preceding sum by
	\begin{align*}
		&\bar{\sigma}^2 \sum_{k = 0}^{k_0} 2^{-k} \mathbb{E}\left[ e^{\theta (Z_n + Z_n^t)/2}  \mathbf{1}\left\{|\sigma_{II^t}|/\bar{\sigma}^2  \geq r_0 \right\} \right]  + \bar{\sigma}^2 \hspace{-5pt}\sum_{k= k_0+1}^\infty 2^{-k}  \mathbb{E}\left[ e^{\theta (Z_n + Z_n^t)/2} \mathbf{1}\left\{2^{-k-1} \leq |\sigma_{II^t}|/\bar{\sigma}^2  \leq 2^{-k} \right\} \right] \nonumber\\
		&\leq 2 \bar{\sigma}^2\sum_{D \in \mathcal{C}(r_0)} \mathbb{E}\left[ e^{\theta (Z_n + Z_n^t)/2} \mathbf{1}\left\{ I, I^t \in D\right\}\right] + r_0\mathbb{E}\left[ e^{\theta (Z_n + Z_n^t)/2}\right].
	\end{align*}
	By the hypercontractivity of $Z^t$ and the arguments in~\cite{tanguy2015some} on pp. 242--243, we can upper bound the two sums in above display by
	\begin{align*}
		\left(  2 \bar{\sigma}^2  \nu(r_0) \rho(r_0)^{\tanh(t/2)} + r_0\right) \mathbb{E}\left[ e^{\theta Z_n}\right].
	\end{align*}
	
	Plugging this upper bound into~\eqref{eq:lemma:Tanguy2015some-Heteroscedastic-1} and integrating out over $t \geq 0$ yields the desired bound on the variance. The exponential tail bound now follows from Corollary 3.2 in~\cite{ledoux2001concentration}.	
\end{proof}

\begin{proof}[Proof of Proposition~\ref{theorem:UpperBoundVariance-SeparableProcess}]
	
	Since $X = \{X_u : u \in U\}$ is a separable Gaussian process, the supremum $Z = \sup_{u \in U} X_u$ is measurable. Hence, by Proposition 1.9 in~\cite{ledoux2001concentration} and the Borell-TIS inequality, $\mathrm{Var}(Z) \leq 4 \bar{\sigma}^2$. In the following we refine this upper bound. 	
	Consider the case $\mathbb{E}[Z/\bar{\sigma}] \leq 3 \sqrt{\log 2}$. Then, the statement of the theorem reads $\mathrm{Var}(Z) \leq  4\bar{\sigma}^2 \wedge \infty$, which is trivially true. Therefore, we only need to establish the claim for the case $\mathbb{E}[Z/\bar{\sigma}] > 3 \sqrt{\log 2}$.
	
	By separability of $X$ there exists a sequence of finite sets  $U_n \subseteq U$ with $|U_n|=n$ such that $Z_n:= \max_{u \in U_n} X_u \rightarrow \sup_{u \in U} X_u = Z$ a.s. as $n \rightarrow \infty$. Without loss of generality we may assume that $(U_n)_{n \geq 1}$ is a non-decreasing sequence of subsets of $U$ in the sense that $U_n \subseteq U_{n+1}$ for all $n \geq 1$. Therefore,  by the monotone convergence theorem, $\mathbb{E}[Z_n] \uparrow \mathbb{E}[Z]$, and, since $(\mathbb{E}[Z/\bar{\sigma}])^2 \geq 3^2 \log 2$, there exists $n_0 \geq 1$ such that $(\mathbb{E}[Z_n/\bar{\sigma}])^2 > 8 \log 2$ for all $n \geq n_0$.
	
	Fix such an $n \geq n_0$. Set $r_0 = \bar{\sigma}^2 \rho$ and $\mathcal{C}(r_0) = \{\{1\}, \ldots, \{n\}\}$. Then $\mathcal{C}(r_0)$ is a covering of the index set $[n]$ Furthermore, since $\sigma_{ij} \geq r_0$ only if $i=j$ (by Cauchy-Schwarz), we conclude that for all indices $i,j \in [n]$ with $\sigma_{ij} \geq r_0$ there exists indeed a $D \in \mathcal{C}(r_0)$ with $i,j \in D$. Also, $\sum_{D \in \mathcal{C}(r_0)} \mathbf{1}\{I \in D\} = 1$ by construction of $\mathcal{C}(r_0)$. Thus, the covering $\mathcal{C}(r_0)$ satisfies the conditions of Lemma~\ref{lemma:Tanguy2015some-Heteroscedastic} with $\nu(r_0) = 1$.  
	
	By Chernoff's bound for Gaussian random variables and the Borell-TIS inequality,
	\begin{align*}
		\mathbb{P}\{ I = i\} = \mathbb{P}\{ X_i = Z_n\} &= \mathbb{P}\{ X_i = Z_n, Z _n> \mathbb{E}[Z_n]/2\} + \mathbb{P}\{ X_i = Z_n, Z_n \leq \mathbb{E}[Z_n]/2\}\\
		&\leq  \mathbb{P}\{ X_i > \mathbb{E}[Z_n]/2\} + \mathbb{P}\{ Z_n - \mathbb{E}[Z_n] \leq -\mathbb{E}[Z_n]/2\}\\
		&\leq 2 e^{-(\mathbb{E}[Z_n/\bar{\sigma}])^2/8}.
	\end{align*} 
	Since the expression on the far right hand side in above display is independent of index $i$, it serves as an upper bound on $\rho(r_0) =\max_{D \in \mathcal{C}(r_0)} \mathbb{P}\{  I \in D\} = \max_{1 \leq i \leq n}  \mathbb{P}\{  I = i\}$. Whence, by Lemma~\ref{lemma:Tanguy2015some-Heteroscedastic}, for all $n \geq n_0$ and $t \geq 0$,
	\begin{align}\label{eq:lemma:UpperBoundVariance-SeparableProcess-1}
		\mathbb{P} \left\{  |Z_n - \mathbb{E}[Z_n]| \geq t\right\} \leq 6e^{-t/\sqrt{K_n}},
	\end{align}
	where $K_n = 3^2 \bar{\sigma}^2\rho + \frac{12^2\bar{\sigma}^2}{(\mathbb{E}[Z_n/\bar{\sigma}])^2/8 - \log 2}$.
	
	Since $Z_n \rightarrow Z$ a.s. and $\mathbb{E}[Z_n] \uparrow \mathbb{E}[Z]$, we also have $|Z_n - \mathbb{E}[Z_n]| \rightarrow |Z - \mathbb{E}[Z]|$ a.s. Therefore, there exits an at most countable subset $\mathcal{N} \subset [0, \infty]$ such that for all $t \in [0, \infty] \setminus \mathcal{N}$,
	\begin{align*}
		\lim_{n \rightarrow \infty}  \mathbb{P} \left\{ |Z_n - \mathbb{E}[Z_n]| \geq t\right\} =  \mathbb{P} \left\{|Z - \mathrm{E}[Z]| \geq t\right\}.
	\end{align*}
	Moreover, since $\mathbb{E}[Z_n/\bar{\sigma}] \rightarrow \mathbb{E}[Z/\bar{\sigma}]$, the exponential tail bound in~\eqref{eq:lemma:UpperBoundVariance-SeparableProcess-1} implies
	\begin{align*}
		\lim_{n\rightarrow \infty} \mathbb{P}\left\{|Z_n - \mathrm{E}[Z_n]| \geq t\right\} \leq 6e^{-t/\sqrt{K}},
	\end{align*}
	where $K = 3^2 \bar{\sigma}^2\rho + \frac{12^2\bar{\sigma}^2}{(\mathbb{E}[Z/\bar{\sigma}])^2/8 - \log 2}$. So, by Fubini's theorem,
	\begin{align*}
		\mathrm{Var}\left(Z\right) = \int_0^\infty 2t  \mathbb{P} \left\{|Z - \mathrm{E}[Z]| \geq t\right\} dt \leq  \int_0^\infty 12t e^{-t/\sqrt{K}} dt \leq 108 \bar{\sigma}^2\rho + \frac{12^3\bar{\sigma}^2}{(\mathbb{E}[Z/\bar{\sigma}])^2/8 - \log 2 }.
	\end{align*}
	Since $a^2 - b^2 \geq (a - b)^2$ for all $a > b \geq 0$,
	\begin{align*}
		\mathrm{Var}\left(Z\right) \leq 60^2 \left(\frac{\bar{\sigma}}{(\mathbb{E}[Z/\bar{\sigma}] - 3 \sqrt{\log 2})_+} \right)^2 + 11^2\bar{\sigma}^2\rho.
	\end{align*}
	This completes the proof.	
\end{proof}

\section{Proof of Theorem~\ref{theorem:AntiConcentration-OrderStatistics}}\label{sec:Proof-AntiConcentration}

\subsection{Heuristics}\label{subsec:ProofStrategy-AntiConcentration-OrderStatistics}
Let $\mathcal{L}$ be the Lebesgue measure on $\left(\mathbb{R}, \mathcal{B}(\mathbb{R})\right)$ and $Z \sim P$ be a random variable with density $f_Z = \frac{dP}{d\mathcal{L}}$. We denote the mode (of the density) of $Z$ by
\begin{align*}
	M(Z) = \esssup_{z \in \mathbb{R}} f_{Z}(z).
\end{align*}
The starting point for the proof of Theorem~\ref{theorem:AntiConcentration-OrderStatistics} are the following two observations that relate concentration function, variance, and mode of $Z$~\citep[][]{bobkov2015concentration}: First, for all $\varepsilon > 0$,
\begin{align*} 
	Q(Z, \varepsilon) = \varepsilon M(Z + \varepsilon U)
\end{align*}
where $U \sim \mathrm{Unif}(0,1)$ independent of $Z$. Second, the quantity
\begin{align*} 
	\mathrm{Var}(Z + \varepsilon U) M^2(Z + \varepsilon U)
\end{align*}
is invariant under location and scale transformations of $Z + \varepsilon U$. To derive bounds on the concentration function $Q(Z, \varepsilon)$ it thus suffices to obtain bounds on $\mathrm{Var}(Z + \varepsilon U) M^2(Z + \varepsilon U)$. Since the latter quantity is invariant under re-scaling of $Z + \varepsilon U$, we may assume that $M(Z + \varepsilon U) \equiv 1$. Hence, if for all $\varepsilon > 0$ the law of $Z + \varepsilon U$ belongs to a family $\mathcal{M}$ of absolutely continuous probability measures, then
\begin{align}\label{eq:subsec:Proof-AntiConcentration-OrderStatistics-3}
		\frac{C_1 \varepsilon }{\sqrt{ \mathrm{Var}(Z) + \varepsilon^2/12} } \leq Q(Z, \varepsilon) \leq \frac{C_2 \varepsilon }{\sqrt{ \mathrm{Var}(Z) + \varepsilon^2/12} },
\end{align}
where we have used that $\mathrm{Var}(Z + \varepsilon U) = \mathrm{Var}(Z) + \varepsilon^2/12$ and $C_1, C_2 \geq 0$ satisfy
\begin{align*}
	C_1 &\leq \inf\left\{  \mathrm{Var}(X) : X \sim P, M(X) = 1, P \in \mathcal{M} \right\},\\
	C_2 &\geq \sup\left\{  \mathrm{Var}(X) : X \sim P, M(X) = 1, P \in \mathcal{M} \right\}.
\end{align*}

This reduces the problem of finding bounds on the concentration function of $Z$ to finding bounds on variances that hold uniformly over a particular class of probability measures. To the best of our knowledge, this approach was first developed by~\cite{bobkov2015concentration} in the context of random variables with log-concave distributions. Its success crucially depends on the properties of the law of $Z$ under convolution with a uniform random variable. In the next section, we demonstrate that this approach also succeeds for the class of distributions of order statistics of general $n$-dimensional Gaussian random fields. Since this high-level approach to anti-concentration inequalities is applicable to any statistic, it also opens up the possibility of extending Gaussian approximation and bootstrapping techniques for high-dimensional data to statistics beyond the maximum. We leave this for future research.

\subsection{Details}\label{subsec:Proof-AntiConcentration-OrderStatistics}
We break down the proof of Theorem~\ref{theorem:AntiConcentration-OrderStatistics} in three propositions that, when combined, implement the proof strategy outlined in the preceding section. 

The first proposition is about the density of the order statistics of $n$-dimensional Gaussian random fields with arbitrary correlation structure. When specialized to the maximum of an $n$-dimensional homogeneous Gaussian random field we recover the results by~\cite{ylvisaker1965expected, ylvisaker1968note} and~\cite{chernozhukov2015ComparisonAnti}. For the proof we borrow ideas from the latter paper.

\begin{proposition}\label{proposition:Density-OrderStatistics-Gaussian}
	Let $X = \{X_i : 1 \leq i \leq n\}$ be a Gaussian random field with $\mathbb{E}[X_i] = \mu_i$, $\mathrm{Var}(X_i) = \sigma_i^2 > 0$, and $\mathrm{Cor}(X_i, X_j) < 1$ for all $1 \leq i \neq j \leq n$. Then, the distribution of $X_{(k)}$ is absolutely continuous with respect to the Lebesgue measure $\mathcal{L}$ on $\left(\mathbb{R}, \mathcal{B}(\mathbb{R})\right)$  and has density
	\begin{align*}
		f_{k, n}(z) =  \sum_{i=1}^n\sum_{\substack{J \subseteq [n]\setminus\{i\} \\|J|=n-k}} \sigma_i^{-1}\phi\left( \frac{z- \mu_i}{\sigma_i}\right)\mathbb{P}\left\{
		\begin{array}{lll}
			X_j > z, \:\:\: \forall j \in J,\\
			X_\ell \leq z, \:\:\: \forall \ell \in [n]\setminus (\{i\} \cup J)
		\end{array}
		\Big| \: X_i = z
		\right\},
	\end{align*}
	where $\phi$ denotes the density of the standard normal distribution and $[n] = \{1, \ldots, n\}$.
\end{proposition}
\begin{remark}
	For independent and identically distributed Gaussian random variables $X_i \sim N(\mu, \sigma^2)$, $1 \leq i \leq n$, above density simplifies to the following well-known expression:
	\begin{align*}
		f_{k, n}(z) = (n/\sigma)  { n-1 \choose n-k } \left(1- \Phi\left(\frac{z-\mu}{\sigma}\right)\right)^{(n-1)-(k-1)}\Phi\left(\frac{z-\mu}{\sigma}\right)^{k-1} \phi\left(\frac{z-\mu}{\sigma}\right),
	\end{align*}
	where $\phi$ and $\Phi$ denote the df and cdf of the standard normal distribution.
\end{remark}

\begin{proof}[Proof of Proposition~\ref{proposition:Density-OrderStatistics-Gaussian}]	
	
	The absolute continuity of the distribution of $X_{(k)}$ with respect to the Lebesgue measure $\mathcal{L}$ on the real line follows from the absolute continuity of the marginals of $X= \{X_i : 1 \leq i \leq n\}$ with respect to $\mathcal{L}$ and the upper bound $\mathbb{P}\left\{X_{(k)} \in A\right\} \leq \sum_{k=1}^n\mathbb{P}\left\{X_k \in A\right\} $ for all Borell sets $A \in \mathcal{B}(\mathbb{R})$. Hence, to show that a version of the density of $X_{(k)}$ is given by $f_{k,n}$, it suffices to show that $\lim_{\epsilon\downarrow 0} \epsilon^{-1}\mathbb{P}\{z < X_{(k)} \leq z + \epsilon\} = f_{k,n}(z)$ for a.e. $z \in \mathbb{R}$.
	
	We have, for every $z \in \mathbb{R}$ and $\epsilon > 0$,
	\begin{align*}
		\{z < X_{(k)} \leq z + \epsilon\} = \bigcup_{i=1}^k A_i^{z, \epsilon},
	\end{align*}
	where, for every $1 \leq i \leq k$,
	\begin{align*}
		A_i^{z, \epsilon} := \left\{ \exists I \subseteq [n], |I| = i, \: \exists J \subseteq [n] \setminus I, |J| = n-k : 
		\begin{array}{lll}
			z < X_j \leq z + \epsilon && \mathrm{if}\:\:  j \in I,\\
			X_j > z  && \mathrm{if}\:\: j \in J,\\
			X_j \leq z && \mathrm{if}\:\: j \in [n] \setminus (I \cup J)
		\end{array}\right\}.
	\end{align*}
	Note that the events $A_1^{z, \epsilon}, \ldots A_k^{z, \epsilon}$ are disjoint. Therefore, we now proceed by showing that $\lim_{\epsilon\downarrow 0} \\ \epsilon^{-1}\mathbb{P}\{A_1^{z, \epsilon}\} =f_{k, n}(z)$ and $\mathbb{P}\{A_i^{z, \epsilon}\} = o(\epsilon)$ for $2 \leq i \leq k$. To this end, we expand $A_1^{z, \epsilon}$ as follows:
	\begin{align*}
		A_1^{z, \epsilon} = \bigcup_{i=1}^n \bigcup_{\substack{J \subseteq [n]\setminus\{i\} \\|J|=n-k}} B_1^{z, \epsilon}(i, J), \quad \text{where} \quad B_1^{z, \epsilon}(i, J) :=  \left\{
		\begin{array}{ll}
			z < X_i \leq z + \epsilon,\\
			X_j > z, \:\:\: \forall j \in J,\\
			X_\ell \leq z, \:\:\: \forall j \in [n] \setminus (\{i\} \cup J)
		\end{array} \right\}.
	\end{align*}
	Since the events $B_1^{z, \epsilon}(i, J)$ are disjoint, we have
	\begin{align*}
		\mathbb{P}\{A_1^{z, \epsilon}\}&= \sum_{i=1}^n \sum_{\substack{J \subseteq [n]\setminus\{i\} \\|J|=n-k}} \mathbb{P}\left\{ B_1^{z, \epsilon}(i, J)\right\}\\
		&= \sum_{i=1}^n \sum_{\substack{J \subseteq [n]\setminus\{i\} \\|J|=n-k}} \int_z^{z + \epsilon} \sigma_i^{-1}\phi\left(\frac{u- \mu_i}{\sigma_i}\right) \mathbb{P}\left\{
		\begin{array}{lll}
			X_j > z, \:\:\: \forall j \in J,\\
			X_\ell \leq z, \:\:\: \forall \ell \in [n]\setminus (\{i\} \cup J)
		\end{array}
		\Big| \: X_i = u
		\right\} d u.
	\end{align*}
	
	For $\lim_{\epsilon \downarrow 0}\epsilon^{-1}\mathbb{P}\{A_1^{z, \epsilon}\}$ to be well-defined, we need to show that, for all $i \in [n]$, all $J \subseteq [n] \setminus\{i\}$, $|J| = n-k$, and a.e. $z \in \mathbb{R}$, the map
	\begin{align*}
		p_z(u) := \mathbb{P}\left\{
		\begin{array}{lll}
			X_j > z, \:\:\: \forall j \in J,\\
			X_\ell \leq z, \:\:\: \forall \ell \in [n]\setminus (\{i\} \cup J)
		\end{array}
		\Big| \: X_i = u
		\right\}
	\end{align*}
	is right continuous at $z \in \mathbb{R}$.
	
	Denote by $	W_{j\setminus i}$ the residual of the orthogonal projection of $X_j- \mu_j$ onto $X_i - \mu_i$, i.e.
	\begin{align*}
		W_{j\setminus i} = X_j - \mu_j - \sigma_i^{-2}\mathbb{E}\left[(X_j-\mu_j)(X_i -\mu_i)\right](X_i - \mu_i) \equiv \sigma_j V_j - \sigma_j\mathbb{E}[V_jV_i]V_i,
	\end{align*}
	where $V_j = \sigma_j^{-1}(X_j - \mu_j)$ and $V_i = \sigma_i^{-1}(X_i - \mu_i)$. By construction $ \{W_{j \setminus i}, W_{\ell \setminus i}\}  \ind V_i$ for all $i \in [n]$, $j \in J$, and $\ell \in [n] \setminus (\{i\} \cup J)$. Whence,
	\begin{align}\label{eq:lemma:Density-OrderStatistics-Gaussian-2}
		p_z(u) &=  \mathbb{P}\left\{
		\begin{array}{ll}
			W_{j \setminus i} > z - \mu_j - \sigma_j\mathbb{E}[V_jV_i]V_i,  \:\:\: \forall j \in J,\\
			W_{\ell \setminus i} \leq z - \mu_\ell - \sigma_j\mathbb{E}[V_\ell V_i]V_i, \:\:\:\forall \ell \in [n]\setminus (\{i\} \cup J)
		\end{array}
		\Big| \:  V_i = \sigma_i^{-1}(u - \mu_i) \right\} \nonumber\\
		&=  \mathbb{P}\left\{
		\begin{array}{ll}
			W_{j \setminus i} > z - \mu_j - (\sigma_j/\sigma_i)\mathbb{E}[V_jV_i](u - \mu_i),  \:\:\: \forall j \in J,\\
			W_{\ell \setminus i} \leq z - \mu_\ell - (\sigma_j/\sigma_i)\mathbb{E}[V_\ell V_i](u - \mu_i), \:\:\:\forall \ell \in [n]\setminus (\{i\} \cup J)
		\end{array}\right\}.
	\end{align}
	
	Now, $W_{j\setminus i}$ and $W_{\ell \setminus i}$ are either degenerate with point mass at 0 (which occurs only if $|\mathrm{Cor}(V_j, V_i)| = 1$ and $|\mathrm{Cor}(V_\ell, V_i)| = 1$) or have a non-degenerate, continuous Gaussian distribution. Clearly, if $W_{j\setminus i}$ or $W_{\ell \setminus i}$ is degenerate, $p_z(u)$ is discontinuous at some $u \in \mathbb{R}$. However, since the event in~\eqref{eq:lemma:Density-OrderStatistics-Gaussian-2} is characterized by $n-1$ linear inequalities (in $u \in \mathbb{R}$), there are at most $n-1$ such discontinuity points. Thus, we conclude that for a.e. $z \in \mathbb{R}$, the map $p_z(u)$ is indeed right-continuous at $z \in \mathbb{R}$. Therefore, for a.e. $z \in \mathbb{R}$,
	\begin{align*}
		\lim_{\epsilon \downarrow 0}\epsilon^{-1}P(A_1^{z, \epsilon}) = \sum_{i=1}^n \sum_{\substack{J \subseteq [n]\setminus\{i\} \\|J|=n-k}} \sigma_i^{-1}\phi\left(\frac{z- \mu_i}{\sigma_i}\right) \mathbb{P}\left\{
		\begin{array}{lll}
			X_j > z, \:\:\: \forall j \in J,\\
			X_\ell \leq z, \:\:\: \forall \ell \in [n]\setminus (\{i\} \cup J)
		\end{array}
		\Big| \:  X_i = z
		\right\}.
	\end{align*}
	
	To complete the proof, we need to show that $\mathbb{P}\{A_i^{z, \epsilon}\} = o(\epsilon)$ for $2 \leq i \leq k$. Fix any $2 \leq i \leq k$. The probability of $\mathbb{P}\{A_i^{z, \epsilon}\}$ is bounded by a sum of terms of the form $\mathbb{P}\{z < X_j \leq z + \epsilon, \: \forall j \in J, \: |J| = i\}$. First, suppose that there exist $j, \ell \in J$ such that $\mathrm{Cor}(X_j, X_\ell) =-1$. Then, for every $z \in \mathbb{R}$ and for $\epsilon > 0$ small enough, $\mathbb{P}\{z < X_j \leq z + \epsilon, \: \forall j \in J, \: |J| = i\} \leq \mathbb{P}\{z < X_j \leq z + \epsilon, \: z  < X_\ell \leq z + \epsilon\} =0$. Next, suppose that $|\mathrm{Cor}(X_j, X_\ell)| < 1$ for all $j, \ell \in J$. Then, $(X_j)_{j \in J}$ follows a $|J|$-dimensional multivariate normal distribution. Hence, for every $z \in \mathbb{R}$ and $\epsilon \downarrow 0 $, $\mathbb{P}\{z < X_j \leq z + \epsilon, \: \forall j \in J, \: |J| = i\} = O(\epsilon^i) = o(\epsilon)$. Lastly, note that the (problematic) case in which there exists $J' \subseteq J$, $|J'| \geq 2 \vee (i-1)$ such that $\mathrm{Cor}(X_j, X_\ell) = 1$ for all $j, \ell \in J'$, is excluded by assumption.	
\end{proof}

The second proposition asserts that the order statistics of $n$-dimensional Gaussian random fields have $s$-concave distributions and probability measures (or laws) for any $s < 0$ and that they preserve this property under convolution with an independent uniform random variable. To formulate this result recall the following definitions.

\begin{definition}[$s$-concave probability measure]\label{definition:s-Concavity-Measure}
	Let $s \in [-\infty, \infty]$. A probability measure $P$ on $\left(\mathbb{R}^n, \mathcal{B}(\mathbb{R}^n)\right)$ is $s$-concave if for all non-empty sets $A, B \in \mathcal{B}(\mathbb{R}^n)$ and for all $\theta \in [0,1]$,
	\begin{align}\label{eq:definition:s-Concavity-Measure-1}
		P(\theta A + (1-\theta)B) \geq  \left( \theta P(A)^s + (1-\theta) P(B)^s\right)^{1/s},
	\end{align}
	where the right hand side of~\eqref{eq:definition:s-Concavity-Measure-1} is interpreted as $P(A) \wedge P(B)$ for $s = -\infty$, $P(A)^\theta P(B)^{1- \theta}$ for $s= 0$, and $P(A) \vee P(B)$ for $s= \infty$.
\end{definition}
\begin{definition}[$s$-concave function]\label{definition:s-Concavity-Function}
	Let $s \in [-\infty, \infty]$. A function $f: \mathbb{R}^n \rightarrow [0, \infty]$ is said to be $s$-concave if for all $x, y \in \mathbb{R}^n$ and for all $\theta \in [0,1]$,
	\begin{align}\label{eq:definition:s-Concavity-Function-1}
		f(\theta x + (1-\theta)y) \geq \left( \theta f^s(x) + (1-\theta) f^s(y)\right)^{1/s},
	\end{align}
	where the right hand side of~\eqref{eq:definition:s-Concavity-Function-1} is interpreted as $f(x) \wedge f(y)$ for $s = -\infty$, $f(x)^\theta f(y)^{1- \theta}$ for $s= 0$, and $f(x) \vee f(y)$ for $s= \infty$.
\end{definition}
\begin{definition}[$s$-concave distribution]\label{definition:s-Concavity-Distribution}
	A random variable $X$ is said to have $s$-concave distribution if it has an $s$-concave density with respect to the Lebesgue measure defined on the affine hull of its support.
\end{definition}

\begin{remark}[Log-concave probability measure, function, and distribution]
	The special case of $0$-concavity is more commonly called \emph{log-concavity}. For a comprehensive review of properties of log-concave probability measures, functions, and distributions we refer to~\cite{saumard2014log-concavity}.
\end{remark}

We can now state the proposition:

\begin{proposition}\label{proposition:S-Concavity-Max-Gaussian}
	Let $X = \{X_i : 1 \leq i \leq n\}$ be a Gaussian random field with $\mathbb{E}[X_i] = \mu_i$, $\mathrm{Var}(X_i) = \sigma_i^2 > 0$, and $\mathrm{Cor}(X_i, X_j) < 1$ for all $1 \leq i \neq j \leq n$.  Let $U$ be a uniform random variable on the interval $I \subset \mathbb{R}$ independent of $X$. Then, $X_{(k)}$ and $X_{(k)} + U$ have $s$-concave distributions and $s$-concave laws for any $s \in [-\infty, 0)$.
\end{proposition}

The proof of this proposition requires several auxiliary lemmas. The first lemma clarifies the relation between an $s$-concave probability measure (or law) and an $s$-concave distribution, the other four lemmas are on monotonicity and preservation properties of $s$-concave functions.

\begin{lemma}[Theorem 3.1,~\citeauthor{borell1975convex},~\citeyear{borell1975convex}]\label{lemma:Borell}
	Let $s \in [-\infty, 1/n]$, $\mathcal{L}^n$ be the Lebesgue measure on $\left(\mathbb{R}^n, \mathcal{B}(\mathbb{R}^n)\right)$, and $P$ be a probability measure such that the affine hull of $\mathrm{supp}(P)$ has dimension $n$. $P$ is an $s$-concave probability measure on $\left(\mathbb{R}^n, \mathcal{B}(\mathbb{R}^n)\right)$ if and only if the density $f = \frac{d P}{d\mathcal{L}^n}$ is $\gamma_n(s)$-concave, where
	\begin{align*}
		\gamma_n(s) = \begin{cases}
			-1/n &\mathrm{if} \: s = -\infty\\
			s/(1-sn) &\mathrm{if} \: s \in (-\infty, 1/n)\\
			\infty &\mathrm{if} \: s = 1/n.	
		\end{cases}
	\end{align*}
\end{lemma}

\begin{lemma}\label{lemma:S-Concavity-Monotonicity-in-S}
	Let $f : \mathbb{R}^n \rightarrow [0, \infty]$ be an $s$-concave function. Then, $f$ is also an $r$-concave function for all $-\infty \leq r \leq s \leq \infty$.
\end{lemma}
\begin{proof}[Proof of Lemma~\ref{lemma:S-Concavity-Monotonicity-in-S}]
	Apply Definition~\ref{definition:s-Concavity-Function} and the power mean inequality.
\end{proof}

\begin{lemma}[Theorem 4.4,~~\citeauthor{borell1975convex},~\citeyear{borell1975convex}]\label{lemma:PreservationLogConcavity-Convolution}
	Let $f, g : \mathbb{R}^n \rightarrow [0, \infty]$ be log-concave functions. Then, the convolution $f \ast g$  is a log-concave function, too.
\end{lemma}

\begin{lemma}\label{lemma:PreservationLogConcavity-Function-Product}
	Let $f_1, \ldots, f_N : \mathbb{R}^n \rightarrow \mathbb{R}$ be log-concave functions. Then the product $\prod_{i=1}^N f_i :\mathbb{R}^n \rightarrow \mathbb{R}$ is a log-concave function, too.
\end{lemma}
\begin{proof}[Proof of Lemma~\ref{lemma:PreservationLogConcavity-Function-Product}]
	Trivial.
\end{proof}

\begin{lemma}\label{lemma:S-Concavity-Sum}
	Let $f_1, \ldots, f_n: \mathbb{R} \rightarrow [0,\infty]$ be $s$-concave functions for some $s \in [-\infty, 0)$. Then, $f : = \sum_{i=1}^n f_i$ is an $s$-concave function, too.
\end{lemma}

\begin{proof}[Proof of Lemma~\ref{lemma:S-Concavity-Sum}]
	Let $x,y \in \mathbb{R}$, $\theta \in [0,1]$, and $s \in (-\infty, 0)$ be arbitrary and compute
	\begin{align*}
		f(\theta x + (1-\theta)y) &= \sum_{i=1}^n f_i(\theta x + (1-\theta)y)\\
		&\overset{(a)}{\geq} \sum_{i=1}^n \Big( \theta f_i^s(x) + (1- \theta) f_i^s(y) \Big)^{1/s}\\
		&\overset{(b)}{\geq} \frac{n}{n^{1/s}}\left(\sum_{i=1}^n \theta f_i^s(x) + (1- \theta) f_i^s(y)\right)^{1/s}\\
		&\overset{(c)}{\geq} \frac{n^{1 + 1/s}}{n^{1/s+1}}\left(\theta \left[\sum_{i=1}^n f_i(x)\right]^s  + (1- \theta) \left[\sum_{i=1}^n f_i(y) \right]^s \right)^{1/s}\\
		&= \left( \theta f^s(x) + (1-\theta) f^s(y)\right)^{1/s},
	\end{align*}
	where (a) holds because each $f_i$ is $s$-concave, and (b) and (c) follow from Jensen's inequality and the convexity of the maps $x \mapsto x^{1/s}$ and $x \mapsto x^s$ for $x  \geq 0$ and $s < 0$. The case $s = -\infty$ follows from similar calculations.
\end{proof}

\begin{proof}[Proof of Proposition~\ref{proposition:S-Concavity-Max-Gaussian}]
We begin with establishing the claims about the distribution and the law of $X_{(k)}$.
Denote by $	W_{j\setminus i}$ the residual of the orthogonal projection of $X_j- \mu_j$ onto $X_i - \mu_i$, i.e.
\begin{align*}
	W_{j\setminus i} = X_j - \mu_j - \sigma_i^{-2}\mathbb{E}\left[(X_j-\mu_j)(X_i -\mu_i)\right](X_i - \mu_i) \equiv \sigma_j V_j - \sigma_j\mathbb{E}[V_jV_i]V_i,
\end{align*}
where $V_j = \sigma_j^{-1}(X_j - \mu_j)$ and $V_i = \sigma_i^{-1}(X_i - \mu_i)$. By construction $ \{W_{j \setminus i}, W_{\ell \setminus i}\}  \ind V_i$ for all $i \in [n]$, $j \in J$, and $\ell \in [n] \setminus (\{i\} \cup J)$. By construction $ \{W_{j \setminus i}, W_{\ell \setminus i}\}  \ind V_i$ for all $i \in [n]$, $j \in J$, and $\ell \in [n] \setminus (\{i\} \cup J)$. Hence, the density of $X_{(k)}$ given in Proposition~\ref{proposition:Density-OrderStatistics-Gaussian} can be written equivalently as
\begin{align*}
	&f_{k, n}(z) \\
	&=  \sum_{i=1}^n\sum_{\substack{J \subseteq [n]\setminus\{i\} \\|J|=n-k}} \sigma_i^{-1}\phi\left( \frac{z- \mu_i}{\sigma_i}\right) \mathbb{P}\left\{
	\begin{array}{ll}
		W_{j \setminus i} > z - \mu_j - (\sigma_j/\sigma_i)\mathbb{E}[V_jV_i](u - \mu_i),  \:\:\: \forall j \in J,\\
		W_{\ell \setminus i} \leq z - \mu_\ell - (\sigma_j/\sigma_i)\mathbb{E}[V_\ell V_i](u - \mu_i), \:\:\:\forall \ell \in [n]\setminus (\{i\} \cup J)
	\end{array}\right\}.
\end{align*}
Since $X = (X_1, \ldots, X_n)$ is jointly normally distributed, the collection $W_{\setminus i} := \left\{W_{j\setminus i}: \forall j \neq i\right\}$ is also jointly normally distributed for every $1 \leq i \leq n$. Thus, Lemma~\ref{lemma:Borell} applied to $s= 0$ and $P = \mathbb{P} \circ  W_{ \setminus i}^{-1}$ implies that the maps
\begin{align*}
	z \mapsto  \mathbb{P}\left\{
	\begin{array}{ll}
		W_{j \setminus i} > z - \mu_j - (\sigma_j/\sigma_i)\mathbb{E}[V_jV_i](u - \mu_i),  \:\:\: \forall j \in J,\\
		W_{\ell \setminus i} \leq z - \mu_\ell - (\sigma_j/\sigma_i)\mathbb{E}[V_\ell V_i](u - \mu_i), \:\:\:\forall \ell \in [n]\setminus (\{i\} \cup J)
	\end{array}\right\}, \:\: i \in [n], \:\: J \subset [n] \setminus \{i\},
\end{align*}
are log-concave. Hence, by Lemmas~\ref{lemma:S-Concavity-Monotonicity-in-S} and~\ref{lemma:PreservationLogConcavity-Function-Product} the maps
\begin{align}\label{eq:lemma:S-Concavity-Max-Gaussian-1}
	\begin{split}
		z \mapsto  \sigma_i^{-1}\phi\left( \frac{z- \mu_i}{\sigma_i}\right)  \mathbb{P}\left\{
		\begin{array}{ll}
			W_{j \setminus i} > z - \mu_j - (\sigma_j/\sigma_i)\mathbb{E}[V_jV_i](u - \mu_i),  \:\:\: \forall j \in J,\\
			W_{\ell \setminus i} \leq z - \mu_\ell - (\sigma_j/\sigma_i)\mathbb{E}[V_\ell V_i](u - \mu_i), \:\:\:\forall \ell \in [n]\setminus (\{i\} \cup J)
		\end{array}\right\},
	\end{split}
\end{align}
are $s$-concave for every $s \in [-\infty, 0]$. Thus, by Lemma~\ref{lemma:S-Concavity-Sum}, $X_{(k)}$ has an $s$-concave distribution for all $s \in [-\infty, 0)$. Therefore, Lemma~\ref{lemma:Borell} implies that the law of $X_{(k)}$ 
is $s$-concave for $s \in [-\infty, 0)$, too.

Next, we establish the claims about the density and the law of $X_{(k)} + U$. The density of $X_{(k)}+U$ is just the sum of the convolutions of the log-concave maps in~\eqref{eq:lemma:S-Concavity-Max-Gaussian-1} with the uniform density on the interval $I \subset \mathbb{R}$. Since the uniform density on convex sets is log-concave, Lemma~\ref{lemma:PreservationLogConcavity-Convolution} implies that these convolutions are log-concave functions, too. Hence, by Lemmas~\ref{lemma:S-Concavity-Monotonicity-in-S} and~\ref{lemma:S-Concavity-Sum}, $X_{(k)} + U$ has an $s$-concave distribution for all $s \in [-\infty, 0)$, and, by Lemma~\ref{lemma:Borell}, the law of $X_{(k)} +U$ is $s$-concave for $s \in [-\infty, 0)$.
\end{proof}

The third proposition provides lower and upper bounds on the location and scale invariant quantity $\mathrm{Var}(Z) M^2(Z)$ when $Z$ has an $s$-concave distribution with $s \in (-1/3, 0]$. This is equivalent to finding the absolute constants $C_1, C_2 > 0$ in inequality~\eqref{eq:subsec:Proof-AntiConcentration-OrderStatistics-3} over the collection of absolutely continuous $s$-concave probability measures with mode constrained to one.

\begin{proposition}\label{proposition:S-Concavity-Affine-Invariance}
	Let $Z$ be a random variable with $s$-concave distribution for some $s \in (-1/3, 0]$. Then,
	\begin{align*}
		\frac{1}{12} \leq \mathrm{Var}(Z) M^2(Z)\leq  \frac{4 ( 1 + s)^3}{(1 + 3s)(1 + 2s)^2}.
	\end{align*}
	The lower bound holds for arbitrary, not necessarily $s$-concave distributions.
\end{proposition}
\begin{remark}
	The lower bound is tight and attained by the uniform distribution. The upper bound is loose. For example, if $Z$ has log-concave distribution~\cite{bobkov2015concentration} obtain an upper bound of $1$, and if, in addition, the median of $Z$ is zero,~\cite{bobkov1999isoperimetric} establishes an even sharper upper bound of $1/2$.
\end{remark}

For the proof we borrow ideas from the paper by~\cite{bobkov2015concentration}. We also use the following characterization of $s$-concave probability measures on the real line.

\begin{lemma}[Lemma 2.2,~\citeauthor{bobkov2007large},~\citeyear{bobkov2007large}]\label{lemma:Bobkov}
	Let $s \in [-\infty, 1)$, $\mathcal{L}$ be the Lebesgue measure on real line, and $P$ be a non-degenerate probability measure on $\left(\mathbb{R},\mathcal{B}(\mathbb{R})\right)$. $P$ is an $s$-concave probability measure if and only if $\left(f \circ Q\right)^{1/(1-s)}$ is concave, where $f = \frac{dP}{d\mathcal{L}}$ and $Q(t) := \inf\{x \in \mathbb{R} :  P\left([-\infty, x]\right) \geq t\}$.
\end{lemma}

\begin{proof}[Proof of Proposition~\ref{proposition:S-Concavity-Affine-Invariance}]
	As explained in Section~\ref{subsec:ProofStrategy-AntiConcentration-OrderStatistics} we may assume without loss of generality that $M(Z) = 1$. It therefore suffices to derive lower and upper bounds on the variance of $Z$.
	
	To establish the lower bound, put $H(z) = \mathbb{P}\left\{|Z - \mathbb{E}[Z] | < z\right\}$. Then, $H(0) = 0$ and $H'(z) = f_Z(z +\mathbb{E}[Z]) + f_Z(-z +\mathbb{E}[Z]) \leq 2 M(Z) = 2$. Therefore, $H(z) = \int_0^z H'(z) dz \leq 2 z$ and 
	\begin{align*}
		\mathbb{P}\left\{|Z - \mathrm{E}[Z] | \geq z\right\} = 1-  H(z) \geq 1 - 2z.
	\end{align*}
	Hence, by Fubini-Tonelli,
	\begin{align*}
		\mathrm{Var}(Z) &= \int_0^\infty 2 z \mathbb{P}\left\{|Z - \mathrm{E}[Z] | \geq z\right\} dz \\
		&\geq  \int_0^{1/2} 2z \mathbb{P}\left\{|Z - \mathrm{E}[Z] | \geq z\right\} dz \geq  \int_0^{1/2} 2z(1-2z)dz = \frac{1}{12}.
	\end{align*}
	Since we have not used $s$-concavity in above derivation, this lower bound holds for any statistic $Z$.
	
	To establish the upper bound, define $I(t) = f_Z\left(Q_Z(t)\right)$, where $Q_Z$ denotes the quantile function associated with $Z$, i.e. $Q_Z(t) = \inf\left\{z \in \mathbb{R}: \mathbb{P}\left\{Z \leq z\right\} \geq t\right\}$. Next, recall that
	\begin{align*}
		\mathrm{Var}(Z) = \frac{1}{2} \int_0^1 \int_0^1 \left(Q_Z(t) - Q_Z(u)\right)^2dt du = \frac{1}{2}\int_0^1\int_0^1 \left(\int_t^u \frac{dv}{I(v)}\right)^2dt du.
	\end{align*}
	Thus, to upper bound the variance of $Z$ we only need to find a lower bound on $I$. By Lemma~\ref{lemma:Borell} the premise of $Z$ being a random variable with $s$-concave distribution for some $s \in (-1/3, 0]$ is equivalent to the law of $Z$ being an $s$-concave probability measure for some $s \in (-1/2, 0]$. Thus, Lemma~\ref{lemma:Bobkov} implies that $I^{1/(1-s)}$ is concave on $[0,1]$ and for some $s \in (-1/2, 0]$. By compactness of $[0,1]$, $I^{1/(1-s)}$ attains its maximum at some $\alpha \in [0,1]$ such that $I^{1/(1-s)}(\alpha) = M^{1/(1-s)}(Z) = 1$. Moreover, for all $v \in [0, \alpha]$,
	\begin{align*}
		I^{1/(1-s)}(v) &= I^{1/(1-s)}\left(\frac{v}{\alpha} \alpha + \left(1- \frac{v}{\alpha}\right) 0\right) \\
		&\geq \frac{v}{\alpha} I^{1/(1-s)}(\alpha) + \left(1- \frac{v}{\alpha}\right)I^{1/(1-s)}(0) \geq \frac{v}{\alpha},
	\end{align*}
	and, for all $v \in (\alpha, 1]$,
	\begin{align*}
		I^{1/(1-s)}(v) &= I^{1/(1-s)}\left(\frac{1-v}{1-\alpha} \alpha + \left(1- \frac{1-v}{1-\alpha}\right) 1\right)\\
		&\geq \frac{1-v}{1-\alpha} I^{1/(1-s)}(\alpha) + \left(1- \frac{1-v}{1-\alpha}\right)I^{1/(1-s)}(1) \geq \frac{1-v}{1-\alpha}.
	\end{align*}
	Whence, $I$ admits the pointwise lower bound
	\begin{align*}
		I(v) \geq \left(\frac{v}{\alpha}\right)^{1-s} \wedge \left(\frac{1- v}{1-\alpha}\right)^{1-s}.
	\end{align*}
	Thus, we have
	\begin{align*}
		\left( \int_t^u \frac{dv}{I(v)} \right)^2 &\leq 2\left( \int_t^u \left(\frac{\alpha}{v}\right)^{1-s}\mathbf{1}\{v \leq \alpha \}dv\right)^2 + 2\left(\int_t^u \left(\frac{1-\alpha}{1-v}\right)^{1-s}\mathbf{1}\{v > \alpha \}dv \right)^2\\
		&\leq 2 \left(\int_t^u \left(\frac{1}{v}\right)^{1-s}dv\right)^2 + 2\left(\int_t^u \left(\frac{1}{1-v}\right)^{1-s}dv\right)^2\\
		& \leq \frac{2}{s^2} \left( \left(\frac{1}{t}\right)^{-s} - \left(\frac{1}{u}\right)^{-s}\right)^2 + \frac{2}{s^2} \left( \left(\frac{1}{1- u}\right)^{-s}- \left(\frac{1}{1- t}\right)^{-s} \right)^2.
	\end{align*}
	The expression in the last line is integrable in $(t, u)$ over $[0,1]^2$ since $s > -1/2$ (and this is the sole reason why have restricted the range of $s$!). In particular, we obtain
	\begin{align*}
		\frac{1}{2}\int_0^1\int_0^1 \left(\int_t^u \frac{dv}{I(v)}\right)^2dt du \leq \frac{4}{(2s + 1)(s + 1)^2}.
	\end{align*}
	By Lemma~\ref{lemma:Borell} we can translate this into the following upper bound on the variance of random variables $Z$ with $s$-concave distribution and $s \in (-1/3, 0)$:
	\begin{align*}
		\mathrm{Var}(Z) \leq \frac{4 ( 1 + s)^3}{(1 + 3s)(1 + 2s)^2}.
	\end{align*}
	By continuity above bound continues to hold for $s= 0$. 
\end{proof}

The proof of Theorem~\ref{theorem:AntiConcentration-OrderStatistics} is now straightforward.

\begin{proof}[Proof of Theorem~\ref{theorem:AntiConcentration-OrderStatistics}]
	By Proposition~\ref{proposition:Density-OrderStatistics-Gaussian}, $X_{(k)}$ has a density and hence, for $\varepsilon > 0$ arbitrary,
	\begin{align*}
		Q(X_{(k)}, \varepsilon) = \varepsilon M(X_{(k)}+ \varepsilon U),
	\end{align*}
	where $U \sim U(0,1)$ is independent of $X_{(k)}$. By Proposition~\ref{proposition:S-Concavity-Max-Gaussian}, $X_{(k)} + \varepsilon U$ has $s$-concave distribution for all $s \in[-\infty, 0)$. Thus, upper and lower bounds on the concentration function follow from Proposition~\ref{proposition:S-Concavity-Affine-Invariance}. Choosing $s= -1/6$, yields the (aesthetically pleasing) absolute constant $\sqrt{125/12} \leq \sqrt{12}$.
\end{proof}

\newpage
\section*{Acknowledgement}
Alexander Giessing is supported by NSF grant DMS-2310578.

\newpage
\normalsize
\setcounter{page}{1}

\bibliography{GBA_51_ref}

\end{document}